\documentclass[letterpaper, 10 pt, conference]{ieeeconf}
\IEEEoverridecommandlockouts                              

\usepackage{cite}
\usepackage{amsmath,amssymb,amsfonts}
\usepackage{algorithmic}
\usepackage{graphicx}
\usepackage{textcomp}
\usepackage{arydshln}
\usepackage{algorithm, algorithmic,color}
\newtheorem{theorem}{Theorem}

\newtheorem{lemma}[theorem]{Lemma}
\newtheorem{proposition}[theorem]{Proposition}
\newtheorem{definition}[theorem]{Definition}
\newtheorem{remark}[theorem]{Remark}

\newcommand{\End}{\hfill $\square$}
\newcommand{\minimize}[1]{\underset{#1}{\textrm{minimize}}}

\newcommand{\changed}[1]{\textcolor{black}{#1}}
\newcommand{\TF}[1]{\textcolor{black}{#1}}

\def\BibTeX{{\rm B\kern-.05em{\sc i\kern-.025em b}\kern-.08em
    T\kern-.1667em\lower.7ex\hbox{E}\kern-.125emX}}
\begin{document}
\title{Willems' fundamental lemma for linear descriptor systems \\and its use for data-driven output-feedback MPC}
%
\author{Philipp Schmitz, Timm Faulwasser, and Karl Worthmann
\thanks{K.\ Worthmann gratefully acknowledges funding by the German Research Foundation DFG (WO~2056/6-1).}
\thanks{T.\ Faulwasser is with TU Dortmund University, Germany (e-mail: timm.faulwasser@ieee.org). }
\thanks{P.\ Schmitz and K.\ Worthmann, are with Technische Universität Ilmenau, Germany (e-mail: [philipp.schmitz, karl.worthmann]@tu-ilmenau.de).}}

\maketitle

\begin{abstract}
In this paper we investigate data-driven \TF{predictive} control of discrete-time linear descriptor systems. Specifically, we give a tailored variant of Willems' fundamental lemma, which shows that for descriptor systems the 
non-parametric modelling via a Hankel matrix requires less data \changed{compared to linear time-invariant systems without algebraic constraints}. 
Moreover, we use this description to propose a data-driven framework for optimal control and predictive control of discrete-time linear descriptor systems. \changed{For the latter, we provide a sufficient stability condition for receding-horizon control before} 
we illustrate our findings with an example. 
\end{abstract}


\section{Introduction}
\label{sec:introduction}

Recently, data-driven control---and in particular Willems' fundamental lemma~\cite{WRMDM05}---is subject to substantial research interest. This includes non-parametric system representations for deterministic discrete-time \changed{linear time-invariant (LTI)} systems~\cite{Coulson2019} and \changed{linear parameter-varying (LPV)} systems~\cite{Verhoek21}, stochastic LTI systems~\cite{tudo:pan21c}, as well as extensions to polynomial and non-polynomial nonlinear systems~\cite{Strasser20a,Alsalti21}. These non-parametric representations enable system identification~\cite{Markovsky21a}, control design~\cite{DePersis19}, and also the implementation of \changed{predictive control}~\cite{Berberich20,Dwyer21}.

In the context of modelling of dynamical systems, continuous-time and discrete-time descriptor systems are of tremendous relevance in applications~\cite{Campbell19}. However, system-theoretic analysis as well as controller design for such systems face several challenges which range from existence of solutions~\cite{Trenn13}, stability and controllability~\cite{Dai89}, to feedback design~\cite{Kunkel00a,ReisVoig19}. In the context of \changed{model predictive control (MPC)}, early works on descriptor systems include~\cite{Findeisen00,yonchev2004model,sjoberg2007model}, while more recent results can be found in~\cite{ilchmann2018model,Ilchmann21}. %
However, to the best of the authors' knowledge, only little has been done in terms of data-driven analysis and control of discrete-time descriptor systems. One of the few exceptions is~\cite{He21}, wherein identification and data-driven feedback design are discussed. 

In the present paper, we show that the \changed{behavioral} approach allows \changed{the} 
consideration of linear discrete-time descriptor systems. To this end, we give a variant of the fundamental lemma tailored to such systems. Interestingly, it turns out that---compared to the usual LTI case---the \TF{necessary} amount of data in the Hankel matrix is reduced for \changed{regular} descriptor systems 
while the persistency of excitation requirements for the input signals do not change. 
Moreover, we leverage the developed non-parametric system description to derive a data-driven \changed{predictive control} framework for \changed{LTI} 
descriptor systems. We give a stability proof based on terminal constraints and illustrate the scheme with a numerical example. 

The remainder of the paper is structured as follows: Section~\ref{sec:basics} recalls the basics of discrete-time \changed{linear} descriptor systems such as their representation in quasi-Weierstra\ss{} form as well as specific controllability and observability notions. Section~\ref{sec:lemma}
presents and discusses a fundamental lemma for discrete-time descriptor systems, while Section~\ref{sec:MPC} turns towards data-driven \changed{predictive control} \changed{tailored to this system class}. 
In Section~\ref{sec:example} \changed{our findings are illustrated by an example before conclusions are drawn}  
in Section~\ref{sec:conclusions}.\\

\noindent \textbf{Notation}: $\mathbb{N}_0$, $\mathbb{N}$ denote the natural numbers with and without zero, respectively. Moreover, for two numbers $a,b \in \mathbb{N}_0$ with $a \leq b$, the non-empty interval $[a,b] \cap \mathbb{N}_0$ is denoted by $[a:b]$. %
\changed{The identity and the zero matrix in $\mathbb R^{n\times m}$ are denoted by $I_n$ and $0_{n\times n}$, respectively.} For a matrix $A\in\mathbb R^{m\times n}$ we denote by $\operatorname{rk}(A)$ and $\operatorname{im}(A)$ the rank and the image of $A$\changed{, respectively}. Further, for $k \in \TF{\mathbb{N}}$ let $\operatorname{diag}_k(A)=I_k \otimes A$, where  $\otimes$ denotes the Kronecker product.

\changed{For a \TF{function $f:\Omega \to \Gamma$}, we denote the restriction of $f$ to $\Omega_0 \subset\Omega$ by $f|_{\Omega_0}$. 
Considering a map $f : [t:T-1] \rightarrow \mathbb R^{k}$ with $t< T$, we denote the vectorization of $f$ by 
\begin{equation*}
    \mathbf{f}_{[t,T-1]} \doteq \begin{bmatrix}
    f(t)^\top & 
    \hdots & f(T-1)^\top
    \end{bmatrix}^\top\in \mathbb R^{k(T-t)}
\end{equation*}
and, for $L \in \mathbb{N}$ with 
$L \leq T-t$, the corresponding Hankel matrix~$H_{L} (f_{[t,T-1]}) \in\mathbb R^{kL \times (T-t-L+1)}$ is defined by
\begin{equation}
   H_{L} (\mathbf{f}_{[t,T-1]}) \doteq  \left[\begin{smallmatrix}
        f(t) 
        & \dots & f(T-L)\\
        \vdots & \ddots & \vdots\\
        f(t+L-1) 
        & \dots & f(T-1)
    \end{smallmatrix}\right] \nonumber.
\end{equation}
Given a symmetric positive-definite matrix $Q$ we define the norm $\lVert x\rVert_{Q}\doteq (x^\top Q x)^{1/2}$.}

\section{Basics of Linear Descriptor Systems}\label{sec:basics}

\noindent We consider discrete-time linear descriptor systems
\begin{subequations} \label{sys}
\begin{align}
    \label{sysa}
    \tag{\theequation{}a}
    Ex(t+1) &= Ax(t) + Bu(t), \\
    \label{sysb}
    \tag{\theequation{}b}
    y(t) &= Cx(t) + Du(t),
\end{align}
\end{subequations}
with (consistent) initial condition \TF{$(Ex)(0)=x^0$}, where $A, E \in \mathbb R^{n\times n}$, $B\in\mathbb R^{n\times m}$, $C\in \mathbb R^{p\times n}$, $D\in\mathbb R^{p\times m}$. We  assume that $\det(\lambda E-A)\neq 0$ for some $\lambda\in \mathbb C$, i.e., \changed{regularity of system~\eqref{sysa}}. Particularly, we are  interested in the case where the matrix $E$ is singular, i.e., $\operatorname{rk}(E)<n$. 

We rely on the \changed{behavior} notion given in \cite[Definition 1.3.4]{PoldermanWillmes98}, i.e., the trajectories of the system \eqref{sys} are collected in the \emph{full behavior},
\begin{equation*}
    \mathfrak B_\mathrm{f} \doteq \left\{ (x,u, y):\mathbb N_0\rightarrow \mathbb  R^{n}\times \mathbb R^{m}\times \mathbb R^{p}
    \,\middle|\,\begin{gathered} x,u, y \text{ satisfy~\eqref{sys}}\\\text{ for all }
    t\in\mathbb N_0\end{gathered}\right\}.
\end{equation*}
Further, we consider the input-output trajectories associated to the full behavior, i.e., the so-called \textit{manifest behavior}
\begin{equation} \label{eq:manifestBehav}
    \mathfrak B_{\mathrm m} \doteq  \left\{ (u,y) : \mathbb{N}_0 \rightarrow \mathbb{R}^{m}\times \mathbb{R}^{p} 
    \,\middle|\, \begin{gathered}
        \exists\, x:\mathbb N_0\rightarrow \mathbb R^{n}: \\
        (x,u,y) \in \mathfrak{B_\mathrm{f}}
    \end{gathered}\right\}.
\end{equation}

For $t,T \in \mathbb{N}_0$, $t \leq T$, we denote the restrictions of the behaviors 
to the finite time interval~$[t,T]$ by $\mathfrak B_\mathrm{f}[t,T]\doteq\{b|_{[t,T]}\,|\,  b\in\mathfrak B_{\mathrm f}\}$ and $\mathfrak B_{\mathrm m}[t,T]\doteq\{b|_{[t,T]}\,|\,  b\in\mathfrak B_{\mathrm m}\}$, respectively. The \emph{consistent} initial values of the system \eqref{sys} are collected in 
\begin{equation}
    \mathfrak V\doteq\{x^0\in\mathbb R^n\,|\, \exists (x, u, y)\in\mathfrak B_\mathrm{f} \text{ with } \TF{(Ex)(0)=x^0} \}.
\end{equation}
Since the descriptor system~\eqref{sysa} is regular, there exist \TF{invertible} matrices $P$,~$\changed{S}\in\mathbb R^{n\times n}$ such that
\begin{equation}
\begin{aligned}
    \changed{S}EP &= \begin{bmatrix}
       I_q & 0 \\ 0 & N
    \end{bmatrix},& \changed{S}AP &=\begin{bmatrix}
        A_1 & 0 \\ 0 & I_r
    \end{bmatrix},\\
    \changed{S}B &= \begin{bmatrix}
       B_1 \\ B_2
    \end{bmatrix},&
    CP&=\begin{bmatrix}
        C_1 & C_2   
    \end{bmatrix},
\end{aligned}
\end{equation}
where $N\in \mathbb R^{r\times r}$ is nilpotent with nilpotency index~$s$, and $A_1\in \mathbb R^{q\times q}$, $B_1\in\mathbb R^{q\times m}$, $B_2\in\mathbb R^{r \times m}$, $C_1\in\mathbb R^{p\times q}$, $C_2\in\mathbb R^{p\times r}$ with $q+r = n$,  cf.\ \cite{BergerIlchmannTrenn12} and \cite[Section 8.2]{Dai89}. Upon introduction of the coordinate change $z=P^{-1} x$, system \eqref{sys} can  equivalently be written  in quasi-Weierstraß form, i.e.
\begin{subequations}
\label{qweier}
\begin{align}
    \label{qweiera}
    \begin{bmatrix}
        I_{q} & 0\\ 0 & N
    \end{bmatrix} z(t+1)&= \begin{bmatrix}
        A_1 & 0 \\ 0 & I_{r}\end{bmatrix} z(t)+ \begin{bmatrix}
            B_1 \\ B_2
        \end{bmatrix} u(t)\\[0.5\baselineskip]
        \label{qweierb}
        y(t) & = \begin{bmatrix}
        C_1 & C_2   
    \end{bmatrix} z(t) + D u(t).
\end{align}
\end{subequations}
\changed{Although the quasi-Weierstraß form is not unique, the nilpotency index $s$ and the state dimensions $q$ and $r$ do not depend on the particular transformation matrices $P$ and $S$ \cite[Lemma~2.10]{Kunkel00a}. Put differently, the indices~$q$ and~$s$ (and, thus, $r = n-q$) are invariants of the original system~\eqref{sys} preserved in the quasi-Weierstraß form.} %
Similarly to before, we consider the full and \changed{manifest} behavior as well as their restrictions to finite time intervals for system \eqref{qweier}. Specifically, we denote these behaviors by $\mathfrak B'_\mathrm{f}$, $\mathfrak B'_\mathrm{m}$, $\mathfrak B'_\mathrm{f}[t,T]$, and $\mathfrak B'_\mathrm{m}[t,T]$, respectively. Note that \changed{$\mathfrak B_\mathrm{f}'=\{(z,u,y)\,|\, (Pz, u, y)\in\mathfrak B_\mathrm{f})\}$}. Further, observe that the manifest behaviors of \eqref{sys} and \eqref{qweier} coincide, i.e., $\mathfrak B_\mathrm{m}=\mathfrak B'_\mathrm{m}$ and $\mathfrak B_\mathrm{m}[t,T] = \mathfrak B'_\mathrm{m}[t,T]$.

Given an input trajectory $u:\mathbb N_0\rightarrow\mathbb R^m$ and an initial value $z_1^0\in\mathbb R^q$
there is a unique trajectory $(z,u,y)\in\mathfrak B'_\mathrm{f}$ such that the state 
\begin{equation*}
    z=\begin{bmatrix}
       z_1^\top & z_2^\top
\end{bmatrix}{}^\top:\mathbb N_0\rightarrow\mathbb R^{q+r}
\end{equation*}
satisfies $z_1(0)=z_1^0$. This state~$z(t)$, $t \in \mathbb{N}_0$, is given by
\begin{subequations}
\label{dyn}
\begin{align}
    \label{dyn1}
        z_1(t) &= \phantom{-}A_1^t z_1(0) + \sum_{k=1}^t A_1^{t-k}B_1u(k-1)\\
    \label{dyn2}
    z_2(t) &=  -\sum_{k=0}^{s-1} N^{k} B_2 u(t+k).
\end{align}
\end{subequations}
\TF{Observe that to determine the state $z$ at time $t$ one needs 
the future inputs $u(t),\ldots, u(t+s-1)$. Respectively, the future inputs need to satisfy~\eqref{dyn2}. Hence, system~\eqref{qweier} can be regarded as non-causal.}

The set of consistent initial values of system~\eqref{qweier} in quasi-Weierstraß form is given by \changed{$\mathfrak V'=S\mathfrak V$} and can be equivalently characterized as
\begin{equation}
\label{qweier_cons}
\mathfrak V' = \left\{\begin{bmatrix}
         z_1^0\\
         z_2^0
    \end{bmatrix}\in\mathbb R^{q+r}\,\middle|\, \begin{gathered}\exists u\in [0:\changed{s-2}]\rightarrow\mathbb R^m\text{ s.t.}\\
    z_2^0 = -\sum_{k=0}^{\changed{s-2}} N^{\changed{k+1}} B_2 u(k)\end{gathered}\right\}.
\end{equation}
The characterization~\eqref{qweier_cons} together with the transformation $P$ gives rise to an equivalent description of the set of consistent initial values~$\mathfrak V$ of the original system~\eqref{sys}, cf.\ the concept of an input index in the continuous-time setting~\cite{IlchLebe19}.

Next we recall the concepts of R-controllability and R-observability, established in \cite{Dai89}, see also \cite{BelovAndrianovaKurdyukov18, Stykel02}.
\begin{definition}[R-controllability and R-observability~\cite{Dai89}]
    The descriptor system \eqref{sys} is called \emph{R-controllable} if
    \begin{subequations}
    \begin{equation}
        \operatorname{rk}\left(\begin{bmatrix}
       \lambda E -A & B
\end{bmatrix}\right)=n
    \end{equation}
    holds for all $\lambda\in\mathbb C$.
    System \eqref{sys} is called \emph{R-observable} if
    \begin{equation}
        \operatorname{rk}\left(\begin{bmatrix}
             \lambda E-A\\ C
        \end{bmatrix}\right)=n
    \end{equation}
    \end{subequations}
   holds for all $\lambda\in \mathbb C$.\End
\end{definition}
\begin{remark}[Controllability/Observability conditions] 
\label{r-contr}
\TF{
The R-controllability property is equivalent to the usual Kalman controllability rank condition for $z_1$ in~\eqref{qweiera} 
\begin{subequations}
    \begin{equation}\label{kalman1}
        \operatorname{rk}\left(\begin{bmatrix}
           B_1 & A_1B_1 & \dots & A_1^{q-1}B_1
        \end{bmatrix}\right)=q,
    \end{equation}
    see \cite{BelovAndrianovaKurdyukov18}. Similarly, R-observabilty is equivalent to 
    \begin{equation}\label{kalman2}
        \operatorname{rk}\left(\begin{bmatrix}
           C_1\\
           C_1 A_1\\
           \vdots\\
           C_1 A_1^{q-1}
        \end{bmatrix}\right) = q.
    \end{equation}
\End
\end{subequations}}
\end{remark}

The next lemma provides a lower bound on the length of an input-output trajectory to  guarantee uniqueness of the corresponding internal state. 
\begin{lemma}[Uniqueness of state trajectories]
    \label{state_alignment}
        \TF{Consider system~\eqref{sys}  let the corresponding values of $q$ and $s$ be known. Assume that~\eqref{sys} is R-observable.} If two trajectories $(x,u,y)$,~$(\tilde x, \tilde u,\tilde y)\in\mathfrak B_\mathrm{f}[0,q+s-2]$ satisfy \changed{$u|_{[0, q+s-2]} = \tilde{u}|_{[0, q+s-2]}$ and $y|_{[0, q+s-2]} = \tilde{y}|_{[0, q+s-2]}$, then $x|_{[0, q-1]}=\tilde{x}|_{[0,q-1]}$.}%
\end{lemma}
\begin{proof}
\changed{We consider the corresponding trajectories $(z,u,y)$,~$(\tilde z,\tilde u,\tilde y)\in\mathfrak B_\mathrm{f}'[0,q+s-2]$ of the equivalent system \eqref{qweier}, that is $z=P^{-1}x$, $\tilde z = P^{-1}\tilde x$.} According to  \eqref{dyn} we have
\begin{equation*}
\begin{split}
    &C_1 \Bigl (A_1^t \bigl(z_1(0)-\tilde z_1(0)\bigr) + \sum_{k=1}^t A_1^{t-k}B_1\bigl(u(k-1)-\tilde  u(k-1)\bigr)\Bigr)\\
    & -C_2\sum_{k=0}^{s-1} N^k B_2 \bigl(u(t+k)-\tilde u(t+k)\bigr) + D\bigl(u(t)-\tilde u(t)\bigr)\\
    &= \changed{y(t) - \tilde y(t) = 0}
\end{split}
\end{equation*}
for $t=1,\dots, q-1$,
This implies
\begin{equation}
\label{observability}
   \changed{C_1 A_1^t \bigl(z_1(0)-\tilde z_1(0)\bigr) =0}
\end{equation}
for all $t=0,\dots, q-1$. With \eqref{kalman2} this yields $z_1(0)=\tilde z_1(0)$. Moreover, \eqref{dyn2} implies $z_2(0)=\tilde z_2(0)$.
By evolving the states $z_1$ and $z_2$ via \eqref{dyn} up to the time $q-1$ we find
\changed{$z_1|{}_{[0,q-1]}=\tilde{z}_1|{}_{[0,q-1]}$ and $z_2|{}_{[0,q-1]}=\tilde{z}_2|{}_{[0,q-1]}$.} The assertion follows with $x=Pz$ and $\tilde x = P\tilde z$.
\end{proof}

\section{The Fundamental lemma for descriptor systems} \label{sec:lemma}

\noindent We recall the notion of persistency of excitation. 
\begin{definition}[Persistency of excitation]
    \TF{A function $u:[0:T-1]\rightarrow \mathbb R^{m}$} is said to be \emph{persistently exciting} of order~$L$ if the Hankel matrix $H_{L}(\mathbf{u}_{[0,T-1]})$ has rank $mL$.\End
\end{definition}
Note that $(m+1)L-1 \leq T$ is necessary for persistency of excitation. %
Further, persistent excitation of order $L$ implies persistent excitation of lower order $\tilde L$, $\tilde L\leq L$.

The next result shows that the vector space $\mathfrak B_{\mathrm m}[0,L-1]$ of input-output trajectories with finite-time horizon is spanned by \changed{a} Hankel matrix built from input-output data. \changed{The result is implicitly included in the original fundamental lemma by Willems et al.~\cite{WRMDM05}, whose original proof heavily relies on algebraic concepts and is formulated in behavioral notation. \TF{Based on a result for explicit LTI systems \cite{WDPCT20} we give a proof in terms of state-space descriptions}. 
\TF{This proof allows to deduce further insights, especially regarding the amount of data needed in the Hankel matrix. The basic idea for descriptor systems is that only $(A_1,B_1)$ subsystem of the quasi-Weierstraß form \eqref{qweiera}, which is an explicit LTI system, has to be persistently excited to reconstruct} trajectories.}

\begin{lemma}[Fundamental lemma for descriptor systems]
\label{mainth}
    Suppose that the system~\eqref{sys} is R-controllable and regular. Let $(\bar u, \bar y)\in \mathfrak B_{\mathrm m}[0,T-1]$ such that \changed{$\bar u$} is persistently exciting of order $L+q+s-1$ and $T,L \in \mathbb{N}$ satisfy $(m+1)(L+q+s)-1\leq T$. \changed{Then $(u,y)\in\mathfrak B_\mathrm{m}[0,L-1]$ if and only if there is $\alpha\in\mathbb R^{(m+p)L\times (T-s-L+2)}$ such that}
    \begin{equation}
    \label{dynamic}
        \begin{bmatrix}
            H_L(\bar{\mathbf{u}}_{[0,T-s]}) \\ H_L(\bar{\mathbf{y}}_{[0,T-s]}) 
        \end{bmatrix}\alpha = \begin{bmatrix}
           \mathbf{u}_{[0,L-1]} \\ \mathbf{y}_{[0,L-1]}
        \end{bmatrix}.
    \end{equation}
    \End
\end{lemma}
\begin{proof}
    Without loss of generality, we assume that  system~\eqref{sysa} is given in quasi-Weierstraß form \eqref{qweiera}. The proof proceeds in two steps.
    
    \TF{Step 1}. Consider $\mathcal S \in\mathbb R^{qL\times q}$, $\mathcal T\in\mathbb R^{qL\times mL}$, $\mathcal R\in\mathbb R^{rL\times m(L+s-1)}$
    \begin{gather*}
        \mathcal S = \begin{bmatrix}
           I_q\\  A_1\\  \vdots \\ A_1^{L-1}
        \end{bmatrix},\quad \mathcal T = \begin{bmatrix}
           0 & \dots & 0 & 0\\
           B_1 & \ddots & \ddots & 0\\
           \vdots & \ddots & \ddots & \vdots\\
           A_1^{L-2} B_1 & \dots & B_1 & 0
        \end{bmatrix},\\
        \changed{\mathcal R = \begin{bmatrix}
            B_2  &  \dots &  N^{s-1}B_2 & 0 & \dots & 0\\
            0  & B_2  &  \dots &  N^{s-1}B_2 & \dots & 0 \\
            \vdots & \ddots  & \ddots & & \ddots & \vdots \\
            0 & \dots & 0 & B_2  & \dots &  N^{s-1}B_2
        \end{bmatrix}
        }
    \end{gather*}
      and $\mathcal U\in\mathbb R^{(n+m)L\times (q+m(L+s-1))}$, $\mathcal V\in\mathbb R^{(n+m+p)L\times (n+m)L}$ 
    \begin{align*}
        \mathcal U&\doteq\left[\begin{array}{@{}c:cc @{}}
        \mathcal S & \mathcal T & 0_{qL\times m(s-1)} \\
        0_{rL\times q} & \multicolumn{2}{c}{-\mathcal R}\\
        0_{mL\times q}  & I_{mL} & 0_{mL\times m(s-1)}
        \end{array}\right]\\
        \mathcal V&\doteq\begin{bmatrix}
            I_{qL} & 0_{qL\times rL} & 0_{qL\times mL}\\
           0_{rL\times  qL} & I_{rL} & 0_{rL\times mL} \\
           0_{mL\times  qL} & 0_{mL\times rL} & I_{mL} \\
           \operatorname{diag}_L(C_1) & \operatorname{diag}_L(C_2) & \operatorname{diag}_L(D)
        \end{bmatrix}.
    \end{align*}
    We show that $(z,u,y)\in\mathfrak B'_\mathrm{f}[0,L-1]$ if and only if
    \begin{equation}
    \label{tinte}
        \begin{bmatrix} {\mathbf{z}_1}_{[0, L-1]} \\ {\mathbf{z}_2}_{[0, L-1]}\\ \mathbf{u}_{[0, L-1]} \\ \changed{\mathbf{y}_{[0,L-1]}} \end{bmatrix} \in\operatorname{im}(\mathcal V\mathcal U)
    \end{equation}
    holds, where $z(t)$ is composed of two vectors $z_1(t)\in\mathbb R^{q}$ and $z_2(t)\in\mathbb R^{r}$. \TF{To this end, let} $(z,u,y)\in \mathfrak B'_\mathrm{f}[0,L-1]$. Then there exists $(z^*, u^*, y^*)\in\mathfrak B'_\mathrm{f}$ with $z^*|_{[0,L-1]}=z$, $u^*|_{[0,L-1]}=u$ and $y^*|_{[0,L-1]}=y$. The explicit solution~\eqref{dyn} of~\eqref{qweiera} gives
    \begin{equation*}
        \begin{bmatrix} {\mathbf{z}_1^*}_{[0, L-1]} \\ {\mathbf{z}_2^*}_{[0, L-1]}\\ \mathbf{u}^*_{[0, L-1]} \end{bmatrix}= \mathcal U \begin{bmatrix}
            z_1^*(0) \\ \mathbf{u}^*_{[0,L+s-2]}
         \end{bmatrix},
    \end{equation*}
    which together with \eqref{qweierb} yields
    \begin{equation}
        \label{tinte2}
        \begin{bmatrix} {\mathbf{z}_1}_{[0, L-1]} \\ {\mathbf{z}_2}_{[0, L-1]}\\ \mathbf{u}_{[0, L-1]} \\ \mathbf{y}_{[0,L-1]}\end{bmatrix} = \begin{bmatrix} {\mathbf{z}_1^*}_{[0, L-1]} \\ {\mathbf{z}_2^*}_{[0, L-1]}\\ \mathbf{u}^*_{[0, L-1]} \\ \mathbf{y}^*_{[0,L-1]}\end{bmatrix} = \mathcal  V \mathcal U \begin{bmatrix}
            z_1^*(0) \\ \mathbf{u}^*_{[0,L+s-2]}
         \end{bmatrix}.
    \end{equation}
    This implies \eqref{tinte}.
    
    On the other hand, if \eqref{tinte} holds for some $(z,u,y):[0:L-1]\rightarrow \mathbb R^{n}\times \mathbb R^{m}\times \mathbb R^{p}$, then there exists $(z^*, u^*, y^*)\in \mathfrak B'_\mathrm{f}[0,L-1]$ such that \eqref{tinte2} holds. This implies $z^*|_{[0,L-1]}=z$, $u^*|_{[0,L-1]}=u$ and $y^*|_{[0,L-1]}=y$, which shows $(z,u,y)\in\mathfrak B'_\mathrm{f}[0,L-1]$.

    \TF{Step 2}. Consider  $(\bar u, \bar y)\in\mathfrak{B}'_\mathrm{m}[0,T-1]$. There exists $\bar z=\begin{bmatrix}
           \bar z_1^\top & \bar z_2^\top
    \end{bmatrix}{}^\top: [0:T-1]\rightarrow R^{q+r}$ such that $(\bar z,\bar u, \bar y)\in\mathfrak B'_\mathrm{f}[0,T-1]$. By assumption \TF{$(A_1, B_1)$ from~\eqref{qweiera}} is controllable (cf. Remark~\ref{r-contr}) and \changed{$\bar u$} is persistently exciting of order $L+q+s-1$. As a consequence of  \cite[Thm.~1~(i)]{WDPCT20} 
    \begin{equation*}
        \mathcal H\doteq\begin{bmatrix}
           H_1(\bar{\mathbf{z}}_1{}_{[0,T-L-s+1]})\\
           H_{L+s-1} (\bar{\mathbf{u}}_{[0,T-1]})
        \end{bmatrix},
    \end{equation*}
    where $\mathcal H\in\mathbb R^{(q + m(L+s-1))\times (T-L-s+2)}$, has rank $q + m(L+s-1)$. Therefore, $\operatorname{im}(\mathcal V\mathcal U) = \operatorname{im}(\mathcal V\mathcal U\mathcal H)$.
    
    Similar to \eqref{tinte2} one sees that for the $j$th column of the matrix $\mathcal H$, where $j\in\{0,\dots,T-L-s+1\}$,
    \begin{equation*}
        \begin{bmatrix}
            \bar{\mathbf{z}}_1{}_{[j,j+L-1]}\\
            \bar{\mathbf{z}}_2{}_{[j,j+L-1]}\\
            \bar{\mathbf{u}}_{[j, j+L-1]} \\ \bar{\mathbf{y}}_{[j,j+L-1]}
        \end{bmatrix} = \mathcal V\begin{bmatrix}
               \bar{\mathbf{z}}_1{}_{[j,j+L-1]}\\
               \bar{\mathbf{z}}_2{}_{[j,j+L-1]}\\
               \bar{\mathbf{u}}_{[j,j+L-1]}
        \end{bmatrix}=\mathcal V \mathcal U
        \begin{bmatrix}
               \bar z_1(j)\\
               \bar{\mathbf{u}}_{[j,j+L+s-2]}
        \end{bmatrix}.
    \end{equation*}
    Hence, we have
    \begin{equation*}
       \begin{bmatrix}
            H_{L}(\bar{\mathbf{z}}_1{}_{[0,T-s]})\\
            H_{L}(\bar{\mathbf{z}}_2{}_{[0,T-s]})\\
            H_{L}({\bar{\mathbf{u}}}_{[0,T-s]})\\
            H_{L}({\bar{\mathbf{y}}}_{[0,T-s]})\\
        \end{bmatrix}=\mathcal V\mathcal U\mathcal H.
    \end{equation*}
    Consequently, $(z,u,y)\in\mathfrak B'_\mathrm{f}[0,L-1]$ if and only if
    \begin{equation*}
        \begin{bmatrix} {\mathbf{z}_1}_{[0, L-1]} \\ {\mathbf{z}_2}_{[0, L-1]}\\ \mathbf{u}_{[0, L-1]} \\ \changed{\mathbf{y}_{[0,L-1]}} \end{bmatrix} \in \operatorname{im}(\mathcal V\mathcal U)=\operatorname{im}(\mathcal V\mathcal U\mathcal H) = \operatorname{im}\begin{bmatrix}
            H_{L}(\bar{\mathbf{z}}_1{}_{[0,T-s]})\\
            H_{L}(\bar{\mathbf{z}}_2{}_{[0,T-s]})\\
            H_{L}({\bar{\mathbf{u}}}_{[0,T-s]})\\
            H_{L}({\bar{\mathbf{y}}}_{[0,T-s]})\\
        \end{bmatrix}.
    \end{equation*}
    The assertion follows from the definition of the manifest behavior~\eqref{eq:manifestBehav} and $\mathfrak B'_\mathrm{m}[0,L-1] = \mathfrak B_\mathrm{m}[0,L-1]$.
\end{proof}

\begin{remark}[Upper-bounding the data demand] \label{rem:BndData}
In general, the index $s$ of the nilpotent matrix $N$ and the dimension $q$ of $A_1$ in the quasi-Weierstraß system \eqref{qweier} are unknown. \TF{However, 
\begin{equation*}
    \begin{bmatrix}
            H_{L}({\bar{\mathbf{u}}}_{[0,T-1]})\\
            H_{L}({\bar{\mathbf{y}}}_{[0,T-1]})\\
        \end{bmatrix}\alpha \in \mathfrak B_\mathrm{m}[0,L-1],
\end{equation*}
holds, provided that \changed{$\bar u$} is persistently exciting of order $L+k$, where $k\geq q+s-1$.
An upper bound on $k$ is given by the state dimension $n$ of the original system~\eqref{sys}.}\End
\end{remark}
\begin{remark}[Descriptor systems \changed{can work with} less data]\label{rem:LessData}
In the case the matrix $E$ is \TF{invertible}, i.e. $q=n$, $r=0$, and $s=1$, Lemma~\ref{mainth} coincides with results for LTI systems, see for instance in \cite{WDPCT20}. However, it deserves to be noted that in case of a singular matrix $E$ the input-output trajectories of length $L$ can be reconstructed by the Hankel matrix in \eqref{dynamic} which contains only values of the trajectory $(\bar u, \bar y)\in\mathfrak B_\mathrm{m}[0,T-1]$ up to the time $T-s$, while in the LTI case all values of $(\bar u, \bar y)$ are needed. \changed{This might be exploited for system whose physical interpretation \TF{gives rise to insights on $s$ and $q$.}}\End
\end{remark}
\TF{Moreover, we conjecture that recent results which allow further reduction of the data demand in the Hankel matrix~\cite{Yu21}  carry over to the descriptor setting without major issues. The details are, however, beyond the scope of the present paper.}

\section{Data-driven control for descriptor systems}\label{sec:MPC}

\noindent In this section, we demonstrate the ramifications of Lemma~\ref{mainth} for optimal and predictive control. \changed{Suppose that system \eqref{sys} is R-controllable and R-observable.}

\subsection{Descriptor systems: data-driven optimal control}\label{sub:ocp}

\noindent %
The control objective is to steer the system to the origin in finite time, i.e.,\ until the end of the optimization horizon. Moreover, the input-output trajectory is chosen \changed{such that} a quadratic cost function is minimized. In the successor subsection, we embed this Optimal Control Problem (OCP) into a predictive control methodology.

Given an observed trajectory $(u,y)\in\mathfrak B_\mathrm{m}[t-q-s+1,t-1]$, we consider the OCP
\begin{subequations}
\label{ocp_org}
\begin{equation}
    \label{ocp_org1}
    \minimize{(\hat u,\hat y)}~ 
    \sum_{k=0}^{L-1}  \lVert\hat y(t+k)\rVert_{Q}^2 + \lVert\hat u(t+k)\rVert_{R}^2
\end{equation}
subject to $(\hat u,\hat y)\in\mathfrak B_\mathrm{m}[t-q-s+1, t+L-1]$ and 
\begin{align}
    \label{ocp_org3}
    \begin{bmatrix}
           \hat{\mathbf{u}}_{[t-q-s+1,t-1]} \\ \hat{\mathbf{y}}_{[t-q-s+1,t-1]}
        \end{bmatrix} &= \begin{bmatrix}
           \mathbf{u}_{[t-q-s+1,t-1]} \\ \mathbf{y}_{[t-q-s+1,t-1]}
        \end{bmatrix}, \\
    \label{ocp_org4}
    \changed{\begin{bmatrix}
         \hat{\mathbf{u}}_{[t+L-q-s+1, t+L-1]}\\
         \hat{\mathbf{y}}_{[t+L-q-s+1,  t+L-1]}
    \end{bmatrix}}
    &=\begin{bmatrix}
         0\\0
    \end{bmatrix}
\end{align}
\end{subequations}
\changed{with symmetric positive-definte matrices $Q\in \mathbb R^{p\times p}$ and $R\in\mathbb R^{m\times m}$ in the quadratic stage cost.}
Clearly, \changed{the} terminal equality constraint~\eqref{ocp_org4} can be replaced by a terminal inequality constraint on the control~$\hat{u}$ and the output~$\hat{y}$ or even dropped. 
\changed{In the same way one can formulate an OCP targeting a setpoint $(u^\mathrm{s}, y^\mathrm{s})$. We say $(u^\mathrm{s}, y^\mathrm{s})\in\mathbb R^{m}\times \mathbb R^{p}$ is a stationary setpoint if there is $(u,y)\in:\mathfrak B_\mathrm{m}$ with $u(t)=u^\mathrm{s}$ and $y(t)=y^\mathrm{s}$ for all $t\in\mathbb N_0$. In this setting the stage cost function penalizes the distance to $(u^\mathrm{s}, y^\mathrm{s})$ and the terminal constraint is adapted to $(u^\mathrm{s}, y^\mathrm{s})$.}

The \textit{consistency condition}~\eqref{ocp_org3} ensures that the latent internal states of the true and the predicted trajectory are aligned up to time~$t-1$, cf.\ Lemma~\ref{state_alignment}. In particular, the internal state at time $t-1$ imposes further restrictions on the predicted input signal up to the time $t+s-2$. 
\begin{remark}[Relaxing the consistency condition] 
According to \eqref{observability} in the proof of Lemma~\ref{state_alignment}, the consistency condition~\eqref{ocp_org3}, which ensures \TF{consistency of the latent internal state with input and output}, can be relaxed to
    \begin{equation*}
        \begin{bmatrix}
           \hat{\mathbf{u}}_{[t-\vartheta-s+1,t-1]} \\ \hat{\mathbf{y}}_{[t-\vartheta-s+1,t-1]}
        \end{bmatrix} = \begin{bmatrix}
           \mathbf{u}_{[t-\vartheta-s+1,t-1]} \\ \mathbf{y}_{[t-\vartheta-s+1,t-1]}
        \end{bmatrix},
    \end{equation*}
    if the rank condition 
    \begin{equation*}\label{kalman3}
        \operatorname{rk}\left(\begin{bmatrix}
           C_1\\
           C_1 A_1\\
           \vdots\\
           C_1 A_1^{\vartheta-1}
        \end{bmatrix}\right) = q
    \end{equation*}
    holds with $\vartheta<q$ for the quasi-Weierstraß form~\eqref{qweier}.\End
\end{remark}

Lemma~\ref{mainth} implies that all trajectories contained in the manifest behavior~$\mathfrak B_\mathrm{m}[t-q-s+1,t-1]$ can be parameterised by a Hankel matrix. Hence, assuming that there is an input-output trajectory $(\bar u, \TF{\bar y})\in\mathfrak B_\mathrm{m}[0,T-1]$ such that \changed{$\bar u$} is persistently exciting of order \changed{$L+2(q+s-1)$}, OCP~\eqref{ocp_org} is equivalent to 
\begin{subequations}\label{ocp}
\begin{equation}
    \label{ocp1}
    \minimize{\substack{(\hat u, \hat y):[t-q-s+1: t+L-1]\rightarrow \mathbb R^{m }\times \mathbb R^{p} \\\alpha(t)\in\mathbb R^{T-L-2s-q+3}}} \sum_{k=0}^{L-1} \lVert\hat y(t+k)\rVert_{Q}^2 + \lVert\hat u(t+k)\rVert_{R}^2
\end{equation}
subject to
\begin{align}
    \label{ocp2}
 \begin{bmatrix}
           \hat{\mathbf{u}}_{[t-q-s+1,t+L-1]} \\ \hat{\mathbf{y}}_{[t-q-s+1,t+L-1]}
        \end{bmatrix} &= \begin{bmatrix}
            H_{L+q+s-1}(\bar{\mathbf{u}}_{[0,T-s]}) \\ H_{L+q+s-1}(\bar{\mathbf{y}}_{[0,T-s]})
        \end{bmatrix}\alpha(t),\\
    \label{ocp3}
    \begin{bmatrix}
           \hat{\mathbf{u}}_{[t-q-s+1,t-1]} \\ \hat{\mathbf{y}}_{[t-q-s+1,t-1]}
        \end{bmatrix} &= \begin{bmatrix}
           \mathbf{u}_{[t-q-s+1,t-1]} \\ \mathbf{y}_{[t-q-s+1,t-1]}
        \end{bmatrix},\\
    \label{ocp4}
    \changed{\begin{bmatrix}
         \hat{\mathbf{u}}_{[t+L-q-s+1, t+L-1]}\\
         \hat{\mathbf{y}}_{[t+L-q-s+1,  t+L-1]}
    \end{bmatrix}}
    &=\begin{bmatrix}
         0\\0
    \end{bmatrix}.
\end{align}
\end{subequations}
We summarize our findings in the following proposition.
\begin{proposition}[Equivalence of the OCPs]
    The OCPs~\eqref{ocp_org} and~\eqref{ocp} are equivalent, i.e.,\
    \begin{itemize}
        \item[(a)] OCP~\eqref{ocp_org} is feasible if and only if OCP~\eqref{ocp} is feasible,
        \item[(b)] for every optimal solution $(u^\star,y^\star) \in \mathfrak B_\mathrm{m}[t-q-s+1, t+L-1]$ \changed{of the OCP~\eqref{ocp_org}}, there exists $\alpha^\star(t)\in \mathbb R^{T-L-2s-q+3}$ such that $(u^\star,y^\star,\alpha^\star(t)\changed{)}$ is an optimal solution of OCP~\eqref{ocp},
        \item[(c)] for every optimal solution $(u^\star,y^\star,\alpha^\star(t)\changed{)}$ of OCP~\eqref{ocp}, $(u^\star,y^\star)$ is contained in the manifest behavior \changed{$\mathfrak B_\mathrm{m}[t-q-s+1, t+L-1]$} and optimal for OCP~\eqref{ocp_org}. \End
    \end{itemize}
\end{proposition}
\TF{Observe that the comments made in Remark~\ref{rem:BndData} on the know\-ledge of the nilpotency index $s$, on the dimension $q$ of $A_1$ in the quasi-Weierstraß form~\eqref{qweier}, as well as the principal need for less data (Remark~\ref{rem:LessData}) remain valid in the context of OCP~\eqref{ocp}. }

\subsection{Descriptor systems: \changed{data-driven predictive control}}

\noindent 
In \changed{predictive control} OCP~\eqref{ocp} is solved \changed{at} each time step~$t$ and, for the solution $(u^\star, y^\star, \alpha^\star(t))$, the value~$u^\star(t)$ is applied as new input $u(t)$ to the system~\eqref{sys}.
\TF{For the descriptor system~\eqref{sys} we propose the \changed{predictive control} scheme based on the OCP~\eqref{ocp} as summarized in Algorithm~\ref{algo}.}

\changed{Here, we emphasize that, due to the absence of input constraints and due to R-controllability, the optimization problem with convex objective function and affine constraints has a feasible (and, thus, also an optimal solution) for all consistent initial values if the optimization horizon is sufficiently long, i.e., $L\geq\tilde L + q+s-2$, where $\tilde L=2s+q$. 
Roughly speaking the first $s-1$ time steps of the prediction serve to satisfy the noncausal restrictions established by the consistency condition~\eqref{ocp_org3} (see \eqref{dyn2}), followed by $q+s$ steps to steer the latent state into the origin (see \eqref{dyn} and \eqref{kalman1}). The terminal constraint~\eqref{ocp_org4} guarantees that the latent state is zero on $[t+L-q,t+L-1]$, see Lemma~\ref{state_alignment}. This ensures that every (initially) feasible and, in particular, every optimal solution can be extended, recursively feasible, i.e., feasibility of OCP~\eqref{ocp} at the successor time instant~$t+1$, cf.\ \cite{Mayne00a}.}
Analogously, one may conclude asymptotic stability of the origin---or of an arbitrary controlled equilibrium $(y^\mathrm{s},u^\mathrm{s})$ if the stage cost is suitably adapted, i.e., $\| u - u^\mathrm{s} \|_R^2 + \| y - y^\mathrm{s} \|_Q^2$---w.r.t.\ the \changed{predictive control} closed loop resulting from Algorithm~\ref{algo}. Moreover, note that the terminal equality constraint may be replaced by suitably constructed terminal inequality constraints, see, e.g., \cite{yonchev2004model,sjoberg2007model,ilchmann2018model}.
\begin{proposition}[Recursive feasibility and stability]\label{prop:MPC}
    Let system~\eqref{sys} be R-controllable and R-observable and suppose that $Q$ and $R$ are symmetric positive definite. \changed{Let the prediction horizon $L\geq 2q+3s-2$.}
    Assuming initial feasibility, i.e.,\ feasibility of the OCP~\eqref{ocp} at time~$t=0$, feasibility is ensured for all $t \in \mathbb{N}$. Moreover, the origin is globally asymptotically stable w.r.t.\ the \changed{predictive control} closed loop, whereby the domain of attraction is implicitly characterized by the set of all feasible consistent initial values. \End
\end{proposition}
The proof follows the \changed{usual }arguments~\cite{Mayne00a} and is hence omitted. 
We remark that initial feasibility is guaranteed for consistent initial values at time $t=0$ if the optimization horizon is sufficiently long in view of the assumed R-controllability and R-observability as pointed out in the previous subsection. Moreover, we emphasize that the assertions of Proposition~\ref{prop:MPC} remain valid if control constraints and output constraints are imposed. However, the assumed initial feasibility can then not be simply covered by choosing the prediction horizon~$L$ sufficient long despite the assumed R-controllability and R-observability.

\begin{algorithm}[H]
    \caption{\changed{\textbf{: Data-driven predictive control}\\ 
    \textbf{Input}: horizon~$L$, (pers.\ exciting) input/output data $(\bar u,\bar y)$}}\label{algo}
    \begin{algorithmic}[1]
    \STATE Set $t=0$ 
    \STATE Measure $(u,y)\in\mathfrak B_\mathrm{m}[t-q-s+1,t-1]$
    \STATE Compute $(u^\star, y^\star, \alpha^\star(t))$ to  \eqref{ocp}
    \STATE Apply $u(t)=u^\star(t)$
    \STATE $t\gets t+1$ \quad and\quad goto Step\,$2$
    \end{algorithmic}
\end{algorithm}

\subsection{Numerical example} \label{sec:example}

\noindent We consider system~\eqref{sys} with
\begin{gather*}
E =
    \begin{bmatrix}
        0 & 0 & 1 & 0 \\
        1 & 2 & 0 & 2 \\
        2 & 3 & 1 & 3 \\
        1 & 2 & 0 & 2
    \end{bmatrix}, 
   A =\begin{bmatrix}
        \phantom{-}1 & 1 & 0 & 2\\
        \phantom{-}0 & 2 & 1 & 1\\
        \phantom{-}1 & 4 & 2 & 3 \\
        -1 & 1 & 1 & 0
    \end{bmatrix}, B = \begin{bmatrix}
        -1 \\ \phantom{-}2 \\ \phantom{-}2 \\ \phantom{-}3
    \end{bmatrix},
\end{gather*}
and the output matrices
\begin{equation*}
    C=\begin{bmatrix}
        1 & 2 & 1 & 2\\
        0 & 1 & 0 & 1\\
        1 & 2 & 1 & 1\\
        2 & 2 & 1 & 2
    \end{bmatrix},\quad D=0_{4\times 1}.
\end{equation*}
Via the matrices
\begin{equation*}
    P = \begin{bmatrix}
     \phantom{-}0 & -1 & \phantom{-}0 & \phantom{-}1\\
     -1 & \phantom{-}0 & \phantom{-}1 & \phantom{-}1\\
     \phantom{-}1 & \phantom{-}0 & \phantom{-}0 & -1\\
     \phantom{-}1 & \phantom{-}1 & -1 & -1
    \end{bmatrix},\quad
    S = \begin{bmatrix}
        \phantom{-}0 & -1 & \phantom{-}1 & \phantom{-}0\\
        \phantom{-}1 & \phantom{-}2 & -1 & \phantom{-}0\\
        -1 & -1 & \phantom{-}1 & \phantom{-}0\\
        \phantom{-}0 & \phantom{-}1 & \phantom{-}0 & -1
    \end{bmatrix}
\end{equation*}
the system can be transformed into quasi-Weierstraß form ($s = q = 2$), which allows easily to verify the R-controllability as well as the R-observability via~\eqref{kalman1} and \eqref{kalman2}, cf.\ Remark~\ref{r-contr}.

We apply the \changed{predictive control} 
Algorithm~\ref{algo} with prediction horizon \changed{$L=20$}. For the input-output trajectory $(\bar u, \bar y)\in\mathfrak B_\mathrm{m}[0,T-1]$  with $T=30$, the values of $\bar u$ are drawn independently from a uniform distribution over the interval $[-1,1]$ such that \changed{$\bar u$} is persistently exciting of order \changed{$L+2(q+s-1)=26$}. 
Further, we assume that $R=I_m$ and $Q=I_p$. We want to steer the \eqref{sys} to the setpoints
\begin{align*}
    (u^{\mathrm s,1}, y^{\mathrm s, 1})&=\Bigl(0, \begin{bmatrix}
        20 & 0 & 0 & 20
    \end{bmatrix}^\top\Bigr),\\
    (u^{\mathrm s,2}, y^{\mathrm s,2})&=\Bigl(0, \begin{bmatrix}
        -10 & 0 & 0 & -10
    \end{bmatrix}^\top\Bigr)
\end{align*} one by one. A closed-loop \changed{predictive control} trajectory generated by Algorithm~\ref{algo} is shown in Figure~\ref{fig}.
\begin{figure}[t]
    \centering
    \includegraphics[scale=0.8]{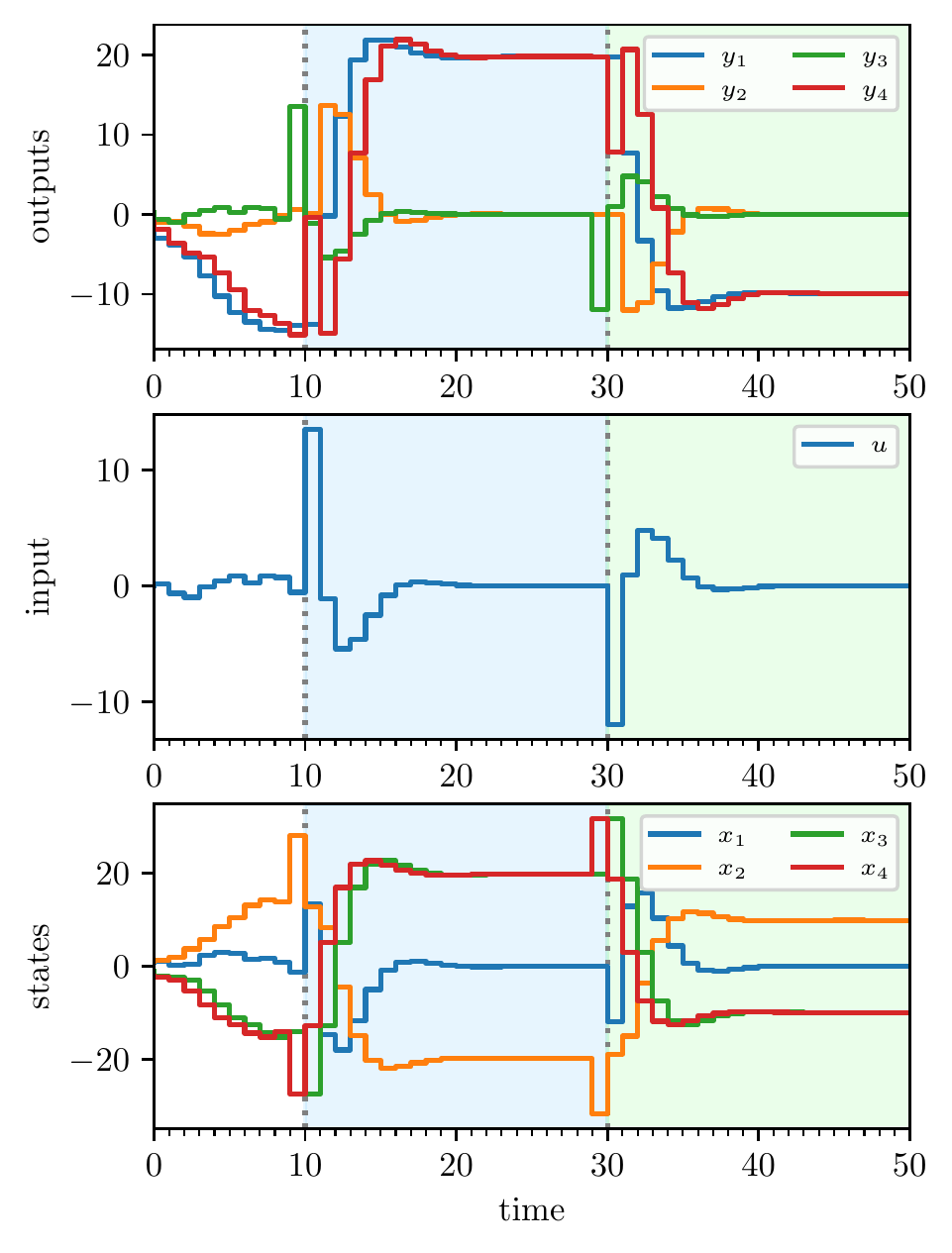}
    \caption{A trajectory emerging from the \changed{predictive control} scheme in Algorithm~\ref{algo}. At time $t=10$ the transition from a random input signal to the optimal control input (blue shaded), which steers the system in to the controlled equilibrium $(u^{\mathrm s,1}, y^{\mathrm s,1})$, can be seen. At time $t=30$ (green shaded) the desired controlled equilibrium is changed to $(u^{\mathrm s,2}, y^{\mathrm s,2})$. \changed{The depicted state $x$ was calculated after the optimization for the sake of illustration.}}
    \label{fig}  
\end{figure} 

\section{Conclusions} \label{sec:conclusions}

\noindent This paper has investigated data-driven control for linear discrete-time descriptor systems. We have shown that---compared to the usual LTI case---in the descriptor setting the data demand for the non-parametric system description via Hankel matrices is reduced. We leveraged Willems' fundamental lemma tailored to descriptor system to propose a data-driven \changed{predictive control} scheme. 
We presented sufficient stability conditions and illustrated the findings with a numerical example. 
Interestingly, in the data-driven \changed{predictive control} setting, and under the considered assumptions, the differences between usual LTI systems and their descriptor counterparts are marginal. This underpins the usefulness of Willems' fundamental and the prospect of data-driven \changed{predictive control}. 

\bibliography{refs}
\bibliographystyle{IEEEtran}

\end{document}